\documentclass[11pt]{amsart}

\usepackage[margin=1.6in]{geometry}

\usepackage{tikz-cd} 
\usepackage{graphicx} 
\usepackage{adjustbox} 
\usepackage{enumerate} 
\usepackage{comment} 

\usepackage{amssymb} 
\usepackage{amsmath} 
\usepackage{mathtools} 
\usepackage{amsfonts} 
\usepackage{amsthm} 
\usepackage{mathrsfs} 
\usepackage{esint}
\usepackage{epigraph} 
\usepackage{calligra}
\usepackage{array}
\usepackage{color}

\usepackage{subfiles} 
\usepackage{appendix} 
\usepackage{pgfplots, tikz, stmaryrd}
\usepackage[english]{babel} 
\usepackage[colorlinks,linkcolor=red,anchorcolor=green,citecolor=blue]{hyperref} 
\hypersetup{linktocpage = true}

\usepackage{xcolor}
\hypersetup{colorlinks,linkcolor={blue!50!black},citecolor={blue!50!black},urlcolor={blue!80!black}}

\usepackage{fouriernc}
\usepackage{fancyhdr}
\usepackage{lastpage}
\usepackage{xstring}
\usepackage{cleveref}

\usepackage{amsthm}

\newenvironment{itenv*}
  {\phantomsection\par\medskip\noindent\itshape}
  {\par\medskip}

\newtheorem{lemma}{Lemma}[section]
\newtheorem{prop}{Proposition}[section]
\newtheorem{thm}{Theorem}[section]
\newtheorem{cor}{Corollary}[section]

\theoremstyle{definition}
\newtheorem{defn}{Definition}[section]
\newtheorem{ex}{Example}[section]

\newtheorem{notation}{Notation}[section]
\newtheorem{conjec}{Conjecture}[section]

\newtheorem{remk}{Remark}[section]

\numberwithin{equation}{section}

\DeclareMathOperator{\SL}{SL} 

\DeclareMathOperator{\Z}{\mathbb{Z}} 

\DeclareMathOperator{\id}{id} 

\DeclareMathOperator{\im}{im} 

\DeclareMathOperator{\sheafhom}{\mathscr{H}\text{\kern -3pt {\calligra\large om}}\,} 
\DeclareMathOperator{\sheafext}{\mathscr{E}\text{\kern -3pt {\calligra\large et}}\,} 
\DeclareMathOperator{\sheafend}{\mathscr{E}\text{\kern -3pt {\calligra\large nd}}\,\,\,} 

\DeclareMathOperator{\can}{can} 

\DeclareMathOperator{\unif}{unif} 

\DeclareMathOperator{\HIG}{HIG} 
\DeclareMathOperator{\MIC}{MIC} 

\DeclareMathAlphabet{\mathcal}{OMS}{cmsy}{m}{n}

\begin{document}
\title{Strong semistability of Higgs bundles over curves}
\author{Bowen Liu}
\address{Yau Mathematical Sciences Center, Tsinghua University, 100084 Beijing, China}
\email{liubw22@mails.tsinghua.edu.cn}

\author{Mao Sheng}
\address{Yau Mathematical Sciences Center, Tsinghua University, 100084 Beijing, China}
\address{Yanqi Lake Beijing Institute of Mathematical Sciences and Applications, 101408, Beijing, China}
\email{msheng@tsinghua.edu.cn}
\begin{abstract}
In this paper we complete the study of the Lan-Sheng-Zuo conjecture proposed in \cite{lan2013semistablehiggsbundlesrepresentations} for the curve case. Precisely, we prove that every nilpotent semistable Higgs bundle of exponent $\leq p-1$ is strongly semistable for curves of genus $g\leq 1$, and over any curves of genus $g\ge2$ construct explicit examples of nilpotent semistable Higgs bundles of exponent $\leq p-1$ and arbitrary big rank (the first example is $p=2$ and rank $3$) which are not strongly semistable. These results are complementary to the strong semistability theorem of Lan-Sheng-Yang-Zuo \cite{MR3994100} and Langer \cite{MR3218782} for semistable Higgs bundles of small rank. 
\end{abstract}

\maketitle
\tableofcontents

\setlength{\parskip}{3pt}

\section{Introduction}
In this paper we work over an algebraically closed field $k$ with $\mathrm{char}(k)=p>0$, and curves mean nonsingular projective curves over $k$.

Let $C$ be a curve and $\mathcal{V}$ be a vector bundle over $C$. Classically, $\mathcal{V}$ is called strongly semistable if for all $n\geq 0$, $(F^n)^*\mathcal{V}$ is semistable, where $F$ is the absolute Frobenius of $C$. It is natural to ask if every semistable vector bundle is strongly semistable. An answer in either direction is interesting to know, since $F$ is inseparable (pullback along a finite seperable morphism preserves semistability). D. Gieseker \cite{MR325616} constructed a rank two semistable vector bundle over curve which is \emph{not} strongly semistable.

Let $X$ be a smooth projective variety defined over $k$ which is $W_2(k)$-liftable. Based on the advance \cite{MR2373230} in non-abelian Hodge theory in positive characteristic, Lan-Sheng-Zuo \cite{lan2013semistablehiggsbundlesrepresentations} introduced the notion of \emph{strong semistability} of Higgs bundles on $X$, and proved that every rank two nilpotent semistable Higgs bundle is strongly semistable (\cite[Theorem 2.6]{lan2013semistablehiggsbundlesrepresentations}). In particular, the example of Gieseker is actually strongly semistable in this new sense, regarded as a Higgs bundle with zero Higgs field. Then it becomes another natural question that whether or not every nilpotent \footnote{The nilpotency condition is a necessary condition for a Higgs bundle to initialize a Higgs-de Rham flow.} semistable Higgs bundle of exponent $\leq p-1$ is strongly semistable. With an optimistic attitude, Lan-Sheng-Zuo proposed the following conjecture:
\begin{conjec}[\protect{\cite[Conjecture 1.5]{lan2013semistablehiggsbundlesrepresentations}}]\label{conj: semistable is strongly semistable}
Every nilpotent semistable Higgs bundle of exponent $\leq p-1$ is strongly semistable.
\end{conjec}

In (\cite[Theorem 1.1]{MR3189770}), Li proved the conjecture for the rank three case over curves under the assumption $\mathrm{char}(k)\ge3$. Accepted as a common terminology nowadays, the rank of a vector bundle is said to be \emph{small}, when it is $\leq p$, and hence will be called \emph{big} in this paper, if the otherwise holds. Hence it remains undetermined whether Li's result can be extended to the first big rank case, namely $p=2$ and rank $3$. On the other hand, the next result extends Li's result to an arbitrary small rank, which shall be called the \emph{strong semistability theorem}.
\begin{thm}[\protect{\cite[Theorem A.1]{MR3994100},~\cite[Theorem 5.5]{MR3218782}}]\label{thm: LSZ conjecture holds for small ranks}
Conjecture \ref{conj: semistable is strongly semistable} holds for small rank Higgs bundles.  
\end{thm} 
\begin{remk}
We remark that the proof of Langer's celebrated theorem \cite[Theorem 2]{MR3314517} on the Bogomolov inequality for semistable Higgs bundles of small rank extends verbatim to strongly semistable Higgs bundles. Hence the Bogomolov inequality also holds for strongly semistable Higgs bundles.  
\end{remk}
The main purpose of this paper is to show that the above results are optimal in the curve case. Our main results are summarized as follows:
\begin{thm}\label{main result}
Let $p$ be any prime number and $C$ be a smooth projective curve over $k$ with $\mathrm{char}(k)=p$. The following three statements hold:
\begin{enumerate}
\item[(1)] If $g(C)\leq 1$, then Conjecture \ref{conj: semistable is strongly semistable} holds true;
\item[(2)] If $g(C)\geq 2$, then for any big rank $r$, there exists a rank $r$ nilpotent semistable Higgs bundle of exponent $\leq 1$ over $C$ which is not strongly semistable;
\item[(3)] Suppose $p$ is odd and $C$ is generic of $g(C)\ge2$, or $p=2$ and $g(C)\ge2$. Then there exist a nilpotent strongly semistable Higgs bundles of exponent $\leq 1$ and that of exponent $=0$ such that their tensor product is not strongly semistable.
\end{enumerate}
\end{thm}
Note that for $C=\mathbb{P}^1$, Conjecture \ref{conj: semistable is strongly semistable} holds trivially, since every semistable Higgs bundle on $\mathbb{P}^1$ is a direct sum of line bundles of the same degree, equipped with the zero Higgs field. We believe $(3)$ remains true when $p$ is odd without assuming that $C$ is generic. This result negates our earlier hope, although the category of semistable Higgs bundles over $C$ does not form a tensor category, the subcategory consisting of strongly semistable Higgs bundles might enjoy this property, making a stronger analogue to the category of semistable Higgs bundles in characteristic zero. 

The fundamental theorem of Deligne-Illusie in Hodge theory in positive characteristic asserts that, for a $W_2(k)$-liftable smooth proper $k$-variety of dimension $\leq p$, the Hodge to de Rham spectral sequence degenerates at $E_1$-page. In a very recent work by A. Petrov \cite[Theorem 1.1]{NondecomposabilityofthedeRhamcomplex}, he constructed a $W_2(k)$-liftable smooth proper $k$-variety of dimension $p+1$ such that the $E_1$-degeneration property fails. Our result may be viewed as an analogue in non-abelian Hodge theory in positive characteristic.

\subsection*{Acknowledgment}
The authors would like to thank Jianping Wang for his comments on Remark \ref{remk: LSZ on higher dimensional cases}, which improves previous result, and to thank Lingguang Li, for his careful reading of the manuscript and for providing valuable comments which improved the paper. They would like to thank heartily an anonymous referee for scrutinizing the work and providing valuable comments and suggestions, which have greatly improved the quality of the paper. The second named author is supported by the Chinese Academy of Sciences Project for Young Scientists in Basic Research (Grant No. YSBR-032).

\section{Preliminaries on Higgs-de Rham flows}
Let $X$ be a nonsingular projective variety over $k$ and $\omega$ be an ample line bundle on $X$. A vector bundle over $X$ in the following means a torsion free coherent sheaf over $X$. 
\begin{defn}
Let $(\mathcal{E},\theta)$ be a Higgs bundle on $X$.
\begin{enumerate}
\item[(1)] We say that $(\mathcal{E}, \theta)$ is \emph{nilpotent of exponent $<\ell$} if for any open subset $U \subset X$ and any collection $\left\{\partial_{x_1}, \ldots,\partial_{x_{\ell}}\right\}$ of sections of $\mathcal{T}_X(U)$, we have $$\theta\left(\partial_{x_1}\right) \ldots \theta\left(\partial_{x_{\ell}}\right)=0.$$
\item[(2)] We say that $(\mathcal{E}, \theta)$ is \emph{nilpotent of exponent $=\ell$} if it is nilpotent of exponent $<\ell+1$ but not nilpotent of exponent $<\ell$ (for $\ell=0$ we require simply that $(\mathcal{E},\theta)$ is nilpotent of exponent $<1$).
\item[(3)] We say that $(\mathcal{E}, \theta)$ is \emph{nilpotent of exponent $\leq\ell$} if it is either nilpotent of exponent $<\ell$ or $=\ell$.
\end{enumerate}
\end{defn}

Let $(\mathcal{V},\nabla)$ be a flat bundle on $X$. Its $p$-curvature $\psi_{\nabla}: \mathcal{V}\to \mathcal{V}\otimes F^*\Omega_X$ is an $F$-Higgs field. By replacing $(\mathcal{E},\theta)$ by $(\mathcal{V},\psi_{\nabla})$ in the previous definition, we arrive at the corresponding definition for $(\mathcal{V},\nabla)$. Namely, it is said to be nilpotent of exponent $<\ell$, if its $p$-curvature is so.  
\begin{notation}
The category consisting of Higgs bundles which are nilpotent of exponent $\leq p-1$ is denoted by $\HIG_{p-1}(X)$, and the category consisting of flat bundles which are nilpotent of exponent $\leq p-1$ is denoted by $\MIC_{p-1}(X)$.
\end{notation}

\begin{lemma}
Let $(\mathcal{E},\theta)$ be a Higgs bundle on $X$. Then the following statements are equivalent:
\begin{enumerate}
\item[(1)] $(\mathcal{E}, \theta)$ is nilpotent of exponent $<\ell$. 
\item[(2)] There exists a filtration of saturated Higgs subsheaves 
$$
0=\mathcal{F}^m \subset \mathcal{F}^{m-1} \subset \ldots \subset \mathcal{F}^0=(\mathcal{E}, \theta)
$$ 
of length $m \leq \ell$ such that the associated graded Higgs bundle has zero Higgs field.
\end{enumerate}
\end{lemma}
\begin{proof}
For (1) to (2): Note that a Higgs bundle $(\mathcal{E},\theta)$ is equivalent to a $\mathrm{Sym}^{\bullet}\mathcal{T}_X$-module $\mathcal{E}$, and we denote $\theta^i(\mathcal{E}):=\im\{\mathrm{Sym}^i\mathcal{T}_X\otimes\mathcal{E}\to\mathcal{E}\}, i\geq 0$, which is a coherent subsheaf of $\mathcal{E}$. Since $(\mathcal{E},\theta)$ is nilpotent of exponent $<\ell$, there exists a unique $m\leq\ell$ such that $\theta^m(\mathcal{E})=0$ and $\theta^{m-1}(\mathcal{E})\neq0$. In particular, we have $0\neq\theta^{m-1}(\mathcal{E})\subseteq\ker\theta$. Now we prove $\ker\theta$ is saturated: Let $\pi\colon\mathcal{E}\to\mathcal{E}/\ker\theta$ be the projection. Then $\widetilde{\ker\theta}=\pi^{-1}\left((\mathcal{E}/\ker\theta)_{\text{tor}}\right)$ is the saturation of $\ker\theta$, which is a Higgs subbundle as the quotient Higgs field preserves $(\mathcal{E}/\ker\theta)_{\text{tor}}$. Since $(\widetilde{\ker\theta},\theta|_{\widetilde{\ker\theta}})$ coincides with $(\ker\theta,0)$ at the generic point, we have $\theta|_{\widetilde{\ker\theta}}=0$. This shows $\widetilde{\ker\theta}\subseteq\ker\theta$, and thus $\ker\theta$ is saturated. 

This gives an exact sequence of Higgs bundles
$$
0\to(\ker\theta,0)\to(\mathcal{E},\theta)\to(\mathcal{E}_1,\theta_1)\to0,
$$
and $(\mathcal{E}_1,\theta_1)$ is nilpotent of exponent $<\ell-1$. Otherwise $\theta^{\ell}(\mathcal{E})/\ker\theta=\theta_1^{\ell-1}(\mathcal{E}_1)\neq0$, a contradiction. Then by induction on $\ell$ this completes the proof of (1) to (2).

For (2) to (1): Let $U\subset X$ be a nonempty open subset and $\{\partial_{x_1},\dots,\partial_{x_{\ell}}\}$ be any collection of sections of $\mathcal{T}_X(U)$. Since the graded Higgs field on $\mathcal{F}^i/\mathcal{F}^{i+1}$ vanishes, it follows that $\theta\left(\partial_{x_j}\right)v\in \mathcal{F}^{i+1}$ for any local section $v\in \mathcal{F}^{i}$. Thus for any local section $v$ of $\mathcal{F}^0=\mathcal{E}$, we have $\theta\left(\partial_{x_{\ell}}\right)v$ is a local section of $\mathcal{F}^1$, and $\theta\left(\partial_{x_{\ell-1}}\right)\theta\left(\partial_{x_{\ell}}\right)v$ is a local section of $\mathcal{F}^{2}$, and so on. Therefore, $\theta\left(\partial_{x_1}\right) \ldots \theta\left(\partial_{x_{\ell}}\right)v$ lies in $\mathcal{F}^m=0$ as long as $\ell\geq m$. That is, $\theta\left(\partial_{x_1}\right) \ldots \theta\left(\partial_{x_{\ell}}\right)=0$ if $\ell\ge m$.
\end{proof}
\begin{remk}
The same statement holds for flat bundles by replacing Higgs fields by $p$-curvatures.
\end{remk}
A. Ogus and V. Vologodsky established the non-abelian Hodge correspondence in positive characteristic, and a special case of it can be stated as follows:
\begin{thm}[\protect{\cite[Theorem 2.8]{MR2373230}}]
If $X$ is $W_2(k)$-liftable, then there is an equivalence of categories between $\HIG_{p-1}(X)$ and $\MIC_{p-1}(X)$, where the correspondence is given by the inverse Cartier transform.
\end{thm}
\begin{remk}
The inverse Cartier transform of a vector bundle $\mathcal{E}$ is the Frobenius pullback $F^*\mathcal{E}$. If the Higgs field is non-trivial, the inverse Cartier transform can be explained as an exponential twisting of Frobenius map by Higgs field in \cite{MR3350108}.
\end{remk}

From now on, $X$ is assumed to be $W_2(k)$-liftable.
\begin{defn}
Let $(\mathcal{E},\theta)$ be a nilpotent Higgs bundle on $X$ of exponent $\leq p-1$. A \emph{Higgs-de Rham flow} with initial term $(\mathcal{E},\theta)$ is an infinite sequence
\begin{center}
\begin{tikzcd}
                                                                     & {(\mathcal{V}_0,\nabla_0)} \arrow[rd, "\mathrm{Gr}_{\mathcal{N}_0}"] &                                                 & {(\mathcal{V}_1,\nabla_1)} \arrow[rd, "\mathrm{Gr}_{\mathcal{N}_1}"] &       \\
{(\mathcal{E}_0,\theta_0)=(\mathcal{E},\theta)} \arrow[ru, "C^{-1}"] &                                         & {(\mathcal{E}_1,\theta_1)} \arrow[ru, "C^{-1}"] &                                         & \dots
\end{tikzcd}
\end{center}
in which $C^{-1}$ is the inverse Cartier transform with respect to a given $W_2(k)$-lifting $X\hookrightarrow\widetilde{X}$ and $\mathcal{N}_i^{\bullet}$ is a Griffiths transverse filtration of $(\mathcal{V}_i,\nabla_i)$ such that $(\mathcal{E}_{i+1},\theta_{i+1})$ is the associated graded Higgs bundle for each $i$ which is nilpotent of exponent $\leq p-1$.
\end{defn}
\begin{defn}
A Higgs-de Rham flow is called an \emph{$\omega$-semistable Higgs-de Rham flow}, if each $(\mathcal{E}_i,\theta_i)$ is an $\omega$-semistable Higgs bundle. 
\end{defn}
\begin{defn}
A semistable Higgs bundle is called \emph{strongly $\omega$-semistable}, if it initializes an $\omega$-semistable Higgs-de Rham flow.
\end{defn}
\begin{ex}
If $\mathcal{E}$ is a strongly $\omega$-semistable vector bundle, then $(\mathcal{E},0)$ is a strongly $\omega$-semistable Higgs bundle since it initializes the following $\omega$-semistable Higgs-de Rham flow
\begin{center}
\begin{tikzcd}
                                                                     & {(F^*\mathcal{E},\nabla_{\can})} \arrow[rd, "\mathrm{Gr}_{\mathcal{N}_0}"] &                                                 & {((F^2)^*\mathcal{E},\nabla_{\can})} \arrow[rd, "\mathrm{Gr}_{\mathcal{N}_1}"] &       \\
{(\mathcal{E},0)} \arrow[ru, "C^{-1}"] &                                         & {(F^*\mathcal{E},0)} \arrow[ru, "C^{-1}"] &                                         & \dots,
\end{tikzcd}
\end{center}
where $\mathcal{N}_{i}^{\bullet}$ is the trivial filtration for each $i$. This shows the notion of strong semistability for a Higgs bundle generalizes naturally the notion of strong semistability for a vector bundle.
\end{ex}

From now on, we will omit to mention the fixed ample line bindle $\omega$ when we consider (semi)stability.
\begin{defn}
Let $(\mathcal{V},\nabla)$ be a flat bundle. A \emph{gr-semistable} filtration is a Griffiths transverse saturated filtration on $(\mathcal{V},\nabla)$ such that the associated graded Higgs bundle is semistable. 
\end{defn}

The following existence theorem of gr-semistable filtration on semistable flat bundles is the key to the strong semistability theorem (Theorem \ref{thm: LSZ conjecture holds for small ranks}).
\begin{thm}[\protect{\cite[Theorem A.4]{MR3994100},~\cite[Theorem 5.5]{MR3218782}}]
If $(\mathcal{V},\nabla)$ is a $\nabla$-semistable flat bundle, then there exists a gr-semistable filtration on $(\mathcal{V},\nabla)$.
\end{thm}

In general, gr-semistable filtration on a semistable flat bundle is not unique, but associated graded Higgs bundles with respect to different gr-semistable filtrations are $S$-equivalent.
\begin{lemma}[\protect{\cite[Proposition 6.9]{MR3994100},~\cite[Corollary 5.6]{MR3218782}}]\label{lemma: uniqueness of gr-semistable fil}
If $\mathcal{N}_1^{\bullet}$ and $\mathcal{N}_2^{\bullet}$ are gr-semistable filtrations of a $\nabla$-semistable flat bundle $(\mathcal{V},\nabla)$, then the reflexivizations of the associated graded slope polystable Higgs bundles obtained from the Jordan-H\"older filtrations of the associated graded semistable Higgs bundles are isomorphic. Moreover, if the associated graded Higgs bundle is stable, then the gr-semistable filtration is unique.
\end{lemma}

\begin{defn}
A gr-semistable filtration is called \emph{gr-polystable}, if the associated graded Higgs bundle is polystable.
\end{defn}
\begin{lemma}\label{lemma: polystable representative has minimal exponent}
Let $(\mathcal{V},\nabla)$ be a semistable flat bundle. If there exists a gr-polystable filtration, then the exponent of nilpotency of the associated graded Higgs bundle with respect to gr-polystable filtration is less than or equal to the exponent of nilpotency of the associated graded Higgs bundle with respect to any gr-semistable filtration.
\end{lemma}
\begin{proof}
Let $(\widetilde{\mathcal{E}},\widetilde{\theta})$ be the associated graded Higgs bundle with respect to a gr-polystable filtration and $(\mathcal{E},\theta)$ be the associated graded Higgs bundle with respect to any gr-semistable filtration. Let 
$$
0=(\mathcal{E}_0,\theta_0)\subset(\mathcal{E}_1,\theta_1)\subset\dots\subset(\mathcal{E}_m,\theta_m)=(\mathcal{E},\theta)
$$
be a Jordan-H\"older filtration of $(\mathcal{E},\theta)$ and let $(\overline{\mathcal{E}}_i,\overline{\theta}_i)=(\mathcal{E}_i,\theta_i)/(\mathcal{E}_{i-1},\theta_{i-1})$ for $1\leq i \leq m$. Then the maximum among the exponents of nilpotency $(\overline{\mathcal{E}}_i,\overline{\theta}_i)$ is less than for equal to the exponent of nilpotency of $(\mathcal{E},\theta)$. 

Moreover, $(\widetilde{\mathcal{E}},\widetilde{\theta})^{**}\cong\bigoplus_{i=1}^m(\overline{\mathcal{E}}_i,\overline{\theta}_i)^{**}$ (where $**$ denotes the reflexivization) by Lemma \ref{lemma: uniqueness of gr-semistable fil}, and this completes the proof because the reflexivization does not change the exponent of nilpotency.
\end{proof}

One is tempted to ask whether one can always find a gr-polystable filtration. Over $\mathbb{P}^1$, any gr-semistable filtration is automatically gr-polystable, since every semistable Higgs bundle on $\mathbb{P}^1$ is polystable. However, this is not always possible when genus $\ge1$. We include the next two examples for their independent interest.
 
\begin{ex}
Let $E$ be an elliptic curve over $k$ with $\mathrm{char}(k)=2$. Let $\mathcal{E}$ be the non-trivial extension of structure sheaf $\mathcal{O}_E$ by a degree one line bundle. Then $\mathcal{E}$ is stable while $F^*\mathcal{E}$ is semistable, as every semistable bundle on elliptic curve is strongly semistable. However, $F^*\mathcal{E}$ is not stable,
since there is no stable bundle of rank $2$ and degree $2$ on elliptic curve.

Suppose a Jordan-H\"older filtration of $F^*\mathcal{E}$ is given by 
$$
0=\mathcal{E}_2\subset\mathcal{E}_1\subset\mathcal{E}_0=F^*\mathcal{E}.
$$
The only possible gr-semistable filtrations of $(F^*\mathcal{E},\nabla_{\can})$ are given by its Jordan-H\"older filtration or the trivial filtration.
We know that $F^*\mathcal{E}$ is indecomposable by \cite[Theorem 2.16]{MR318151} and that $\mathcal{E}_1$ is not $\nabla_{\can}$-invariant. This shows that in any case the associated graded Higgs bundle is not polystable.
\end{ex}
\begin{ex}
Let $C$ be a genus $g\geq 2$ curve and $\mathcal{V}$ be a non-trivial extension of $\mathcal{O}_C$ by $\mathcal{O}_C$. Since $\mathcal{V}$ is an indecomposable bundle of degree zero, there exists a connection $\nabla$ on it by \cite[Proposition 3.1]{MR2237516}. Without loss of generality, we may assume that $\nabla$ does not preserve the subbundle $\mathcal{O}_C$. 

Indeed, consider the sheaf $\mathcal{A}$ of endomorphisms $\varphi\colon\mathcal{V}\to\mathcal{V}\otimes K_C$ which maps the subbundle $\mathcal{O}_C$ into $\mathcal{O}_C\otimes K_C$. Then $\varphi\mapsto\varphi|_{\mathcal{O}_C}$ gives a map $\mathcal{A}\to K_C$ and its kernel is $\mathcal{V}\otimes K_C$. Hence
$$
\dim H^0(C,\mathcal{A})\leq \dim H^0(C,K_C)+\dim H^0(C,\mathcal{V}\otimes K_C).
$$
Since $\mathcal{V}$ is a non-trivial extension of $\mathcal{O}_C$ by $\mathcal{O}_C$, $\dim H^0(C,\mathcal{V}\otimes K_C)=2g-1$ and so $\dim H^0(C,\mathcal{A})\leq 3g-1$. 

On the other hand, since $\mathcal{V}$ is a non-trivial extension of $\mathcal{O}_C$ by $\mathcal{O}_C$, we see that
$\dim H^0(C, \sheafend(\mathcal{V})) = 2$, and this fact, Serre duality and Riemann-Roch imply
that $\dim H^0(\sheafend(\mathcal{V})\otimes K_C ) = 4g-2 > 3g-1$. So we may take an element $\xi\in H^0(C,\sheafend\,(\mathcal{V})\otimes K_C)$ which does \emph{not} preserve the subbundle $\mathcal{O}_C$. Then $\nabla$ or $\nabla+\xi$ is such a connection as requested. 

Claim that there is no gr-polystable filtration on $\mathcal{V}$. Indeed, note that
$$
0\subset\mathcal{O}_C\subset\mathcal{V}
$$
is a gr-semistable filtration on $\mathcal{V}$, and the associated graded Higgs bundle is $(\mathcal{O}_C\oplus\mathcal{O}_C,\theta)$, where $\theta\colon\mathcal{O}_C\to\mathcal{O}_C\otimes K_C$ is a non-zero Higgs field. It is not polystable for the polystable representative in the $S$-equivalent class is $(\mathcal{O}_C\oplus\mathcal{O}_C,0)$. The other possible gr-semistable filtration on $\mathcal{V}$ is the trivial filtration. Then its associated graded Higgs bundle is simply $(\mathcal{V},0)$, which is neither polystable, as $\mathcal{V} $ is a non-trivial extension. The claim is proved.
\end{ex}

\section{Lan-Sheng-Zuo conjecture on elliptic curves}
\begin{prop}\label{prop: the Lan-Sheng-Zuo conjecture on elliptic curve}
Let $E$ be an elliptic curve over $k$. Then every semistable Higgs bundle $(\mathcal{E},\theta)$ which is nilpotent of exponent $\leq p-1$ is strongly semistable.
\end{prop}
\begin{proof}
Let us denote $(\mathcal{V},\nabla)=C^{-1}(\mathcal{E},\theta)$. Then $(\mathcal{V},\nabla)$ is a $\nabla$-semistable flat bundle since $(\mathcal{E},\theta)$ is a semistable Higgs bundle. Now we claim that the underlying bundle of a $\nabla$-semistable flat bundle is semistable. Assume the contrary. Let
$$
0=\mathcal{V}^{s+1}\subsetneq\mathcal{V}^{s}\subsetneq\dots\subsetneq\mathcal{V}^0=\mathcal{V}
$$
be the Harder-Narasimhan filtration of $\mathcal{V}$. Then $\mathcal{V}^{s}$ is not $\nabla$-invariant since $(\mathcal{V},\nabla)$ is $\nabla$-semistable. Then there is a unique $\ell<s$ such that $\nabla(\mathcal{V}^{s})\not\subseteq\mathcal{V}^{\ell+1}$ but $\nabla(\mathcal{V}^s)\subseteq\mathcal{V}^{\ell}$. Then $\nabla$ induces a non-zero morphism $\varphi\colon\mathcal{V}^s\to\mathcal{G}^{\ell}:=\mathcal{V}^{\ell}/\mathcal{V}^{\ell+1}$, which is a contradiction, as $\mu(\mathcal{V}^s)>\mu(\mathcal{G}^{\ell})$.

Since $\mathcal{V}$ is a semistable bundle, we may choose the trivial filtration on $(\mathcal{V},\nabla)$ as its gr-semistable filtration, and the associated graded Higgs bundle is $(\mathcal{V},0)$. If $\mathcal{V}$ is a semistable bundle, then $(F^*\mathcal{V},\nabla_{\can})$ is $\nabla$-semistable, and by above claim we have $F^*\mathcal{V}$ is a semistable bundle. This shows every semistable bundle on elliptic curve is strongly semistable.

As a consequence, every semistable Higgs bundle $(\mathcal{E},\theta)$ which is nilpotent of exponent $\leq p-1$  initializes a semistable Higgs-de Rham flow as follows
\begin{center}
\begin{tikzcd}
                                       & {(\mathcal{V},\nabla)} \arrow[rd, "\mathrm{Gr}_{\mathcal{N}_0}"] &                                        & {(F^*\mathcal{V},\nabla_{\can})} \arrow[rd, "\mathrm{Gr}_{\mathcal{N}_1}"] &          \\
{(\mathcal{E},\theta)} \arrow[ru, "C^{-1}"] &                                                         & {(\mathcal{V},0)} \arrow[ru, "C^{-1}"] &                                                                   & {\dots,}
\end{tikzcd}
\end{center}
where $\mathcal{N}_{i}^{\bullet}$ is the trivial filtration for each $i$, and thus it is strongly semistable.
\end{proof}
\begin{remk}\label{remk: LSZ on higher dimensional cases}
The proof of Proposition \ref{prop: the Lan-Sheng-Zuo conjecture on elliptic curve} can be generalized to prove a higher dimensional statement: Let $X$ be a smooth projective variety over $k$ such that $\mu_{\omega,\max}(\Omega_X^1)\leq0$. Then every $\omega$-semistable Higgs bundle which is nilpotent of exponent $\leq p-1$ is strongly $\omega$-semistable.
\end{remk}
\section{Semistable but not strongly semistable Higgs bundles}

We first provide a construction of semistable but not strongly semistable Higgs bundle of any big rank, thanks to the following results of X. Sun. 

\begin{thm}[\protect{\cite[Theorem 2.2]{MR2415312}}]\label{thm: stable pushforward under Frobenius is stable}
Let $C$ be a curve of genus $g\ge1$. Then $F_*\mathcal{E}$ is semistable whenever $\mathcal{E}$ is semistable. If $g\ge2$, then $F_*\mathcal{E}$ is stable whenever $\mathcal{E}$ is stable.
\end{thm}
\begin{thm}[\protect{\cite[Corollary 2.4]{MR2673745}}]\label{thm: stronger inequality for slope bound}
Let $C$ be a smooth projective curve of genus $g\ge2$ and $\mathcal{L}$ be a line bundle on $C$. Then, for all proper subbundles $\mathcal{E}\subseteq F_*\mathcal{L}$, we have
$$
\mu(\mathcal{E})-\mu(F_*\mathcal{L})\leq-\frac{p-\mathrm{rk}(\mathcal{E})}{p}(g-1).
$$
\end{thm}

\begin{defn}
Let $\mathcal{E}$ be a vector bundle on $C$. Then the \emph{canonical filtration} on $\mathcal{V}=F^*F_*\mathcal{E}$ is a decreasing filtration
$$
\mathcal{V}^{p}\subset\mathcal{V}^{p-1}\subset\dots\subset\mathcal{V}^0=\mathcal{V},
$$
where $\mathcal{V}^1=\ker\{\mathcal{V}=F^*F_*\mathcal{E}\twoheadrightarrow\mathcal{E}\}$ and
$$
\mathcal{V}^{\ell+1}=\ker\{\mathcal{V}^{\ell}\stackrel{\nabla_{\can}}{\longrightarrow}\mathcal{V}\otimes K_C\to\mathcal{V}/\mathcal{V}^{\ell}\otimes K_C\}.
$$
\end{defn}
\begin{lemma}[\protect{\cite[Theorem 5.3]{MR2231194}}]\label{lemma: Canonical filtration on curves}
Let $C$ be a curve of genus $g$ and $\mathcal{E}$ be a vector bundle on $C$. Then the following properties hold for the canonical filtration of $F^*F_*\mathcal{E}$:
\begin{enumerate}
\item[(1)] $\mathcal{V}^0/\mathcal{V}^1\cong\mathcal{E},\nabla_{\can}(\mathcal{V}^{\ell+1})\subseteq\mathcal{V}^{\ell}\otimes K_C$ for $\ell\ge1$.
\item[(2)] $\mathcal{V}^{\ell}/\mathcal{V}^{\ell+1}\stackrel{\nabla_{\can}}{\longrightarrow}(\mathcal{V}^{\ell-1}/\mathcal{V}^{\ell})\otimes K_C$ is an isomorphism for $1\leq\ell\leq p-1$ and $\mathcal{V}^p=0$.
\item[(3)] If $g\ge2$ and $\mathcal{E}$ is semistable, then the canonical filtration is nothing but the Harder-Narasimhan filtration of $F^*F_*\mathcal{E}$.
\end{enumerate}
\end{lemma}

From now on till end, we assume $g\ge2$. Let $\mathcal{L}$ be a line bundle on $C$ and $\mathcal{E}'=F_*\mathcal{L}$. By Theorem \ref{thm: stable pushforward under Frobenius is stable} and Riemann-Roch theorem, $\mathcal{E}'$ is a stable rank $p$ bundle of degree $\deg\mathcal{L}+(p-1)(g-1)$. Moreover, by Lemma \ref{lemma: Canonical filtration on curves} the canonical filtration of $F^*\mathcal{E}'$ is a Griffiths transverse saturated filtration, and by Lemma \ref{lemma: uniqueness of gr-semistable fil} it is the unique gr-semistable filtration since the associated graded Higgs bundle with respect to the canonical filtration is 
$$
(K_C^{p-1}\oplus K_C^{p-2}\oplus\dots\oplus K_C\oplus\mathcal{O}_C,\bigoplus_{i=1}^{p-1}\theta_i)\otimes(\mathcal{L},0),
$$
which is a stable Higgs bundle by Lemma \ref{lemma: stability of Higgs bundle} below, where $\theta_i\colon K_C^{ i}\to K_C^{ i-1}\otimes K_C$ is an isomorphism for each $i=1,\dots,p-1$.
\begin{lemma}\label{lemma: stability of Higgs bundle}
For each $m\in\Z_{>0}$, the Higgs bundle 
$$
(\mathcal{F}^m,\theta^m)=(K_C^{m}\oplus K_C^{m-1}\oplus\dots\oplus K_C\oplus\mathcal{O}_C,\bigoplus_{i=1}^{m}\theta_i)
$$
is stable, where $\theta_i$ is an isomorphism for each $i=1,\dots,m$.
\end{lemma}
\begin{proof}
By restricting to the generic point, this reduces to a linear algebra problem, that is, a $k(C)$-vector space $V$ of dimension $m+1$ together with a basis $\{e_1,\dots,e_{m+1}\}$ such that 
$$
\theta^m(e_1,\dots,e_{m+1})=(e_1,\dots,e_{m+1})J_{m+1}(0),
$$
where $J_{m+1}(0)$ is the upper triangle Jordan block of order $m+1$ with eigenvalue $0$. If $W$ is a $J_{m+1}(0)$-invariant subspace of dimension $r$, then $J_{m+1}(0)^r=0$ on $W$, and thus $W=\mathrm{span}\{e_{m-r+2},\dots,e_{m+1}\}$, as 
$$
J_{m+1}(0)^r=
\begin{pmatrix}
\mathrm{O}_{(m-r+1)\times(m-r+1)}&\mathrm{I}_{(m-r+1)\times r}\\ 
\mathrm{O}_{r\times(m-r+1)}&\mathrm{O}_{r\times r}
\end{pmatrix}.
$$
As a consequence, the rank $r$ Higgs subbundle of $(\mathcal{F}^m,\theta^m)$ must be the direct sum of its last $r$ factors, whose slope is strictly less than the slope of $\mathcal{F}^m$, and this shows $(\mathcal{F}^m,\theta^m)$ is a stable Higgs bundle.
\end{proof}

For each $0<\ell\leq 2(g-1)$, we choose and then fix a degree $\ell$ line bundle $\mathcal{A}_{\ell}$ such that $H^0(C,K_C\otimes\mathcal{A}_{\ell}^{-1})\neq0$. Fix $x_0\in C(k)$, we take $\mathcal{L}$ in the argument above to be some line bundle $\mathcal{L}_{\ell}$ such that 
$$
\mathcal{A}_{\ell}\otimes\mathcal{L}_{\ell}\otimes K_C^{p-1}\cong\mathcal{O}_C(mpx_0)
$$
for some $m\in\Z$. Then $\mathcal{A}_{\ell}\otimes\mathcal{L}_{\ell}\otimes K_C^{p-1}$ can be equipped with the canonical connection. Pick a non-zero element in $H^0(C,K_C\otimes\mathcal{A}_{\ell}^{-1})$, which defines a non-zero morphism 
$$
\theta\colon\mathcal{A}_{\ell}\otimes\mathcal{L}_{\ell}\otimes K_C^{p-1}\to\mathcal{L}_{\ell}\otimes K_C^{p-1}\otimes K_C\subseteq F^*\mathcal{E}'\otimes K_C.
$$

Set
$$
\mathcal{V}_{\ell}=\mathcal{A}_{\ell}\otimes \mathcal{L}_{\ell}\otimes K_C^{p-1}\oplus F^*\mathcal{E}',
$$
which is equipped with a connection $\nabla_{\ell}$ given by
$$
\begin{pmatrix}
\nabla_{\can}&\theta\\ 
0&\nabla_{\can}
\end{pmatrix},
$$
where $\theta$ is defined as above.

\begin{lemma}\label{lemma: V,nalba is stable}
$(\mathcal{V}_{\ell},\nabla_{\ell})$ is $\nabla$-stable.
\end{lemma}
\begin{proof}
By Riemann-Roch theorem, we know that 
$$
\deg(F^*\mathcal{E}')=p\left((p-1)(g-1)+\deg(\mathcal{L}_{\ell})\right),
$$
and hence 
$$
\deg(\mathcal{V}_{\ell})=(p+2)(p-1)(g-1)+(p+1)\deg(\mathcal{L}_{\ell})+\ell.
$$
Thus
$$
\mu(F^*\mathcal{E}')=(p-1)(g-1)+\deg(\mathcal{L}_{\ell})<\frac{p+2}{p+1}(p-1)(g-1)+\frac{\ell}{p+1}+\deg(\mathcal{L}_{\ell})=\mu(\mathcal{V}_{\ell}).
$$
Now suppose $\mathcal{M}$ is a $\nabla$-invariant proper subbundle of $(\mathcal{V}_{\ell},\nabla_{\ell})$. Consider the composition 
$$
\mathcal{M}\hookrightarrow\mathcal{V}_{\ell}\stackrel{\pi}{\longrightarrow}\mathcal{A}_{\ell}\otimes\mathcal{L}_{\ell}\otimes K_C^{p-1},
$$
where $\pi$ is the natural projection. There are two cases:
\begin{itemize}
    \item[$(1)$] $\pi(\mathcal{M})=0$: In this case we have $\mathcal{M}\subseteq F^*\mathcal{E}'$. By Cartier descent, $(F^*\mathcal{E}',\nabla_{\can})$ is a $\nabla$-stable bundle as $\mathcal{E}'$ is a stable bundle. Hence 
    $$
    \mu(\mathcal{M})<\mu(F^*\mathcal{E}')<\mu(\mathcal{V}_{\ell}).
    $$
    \item[$(2)$] $\pi(\mathcal{M})=\mathcal{A}_{\ell}\otimes \mathcal{L}_{\ell}\otimes K_C^{p-1}$: In this case we  have the following exact sequence:
    $$
0\to\mathcal{N}\to\mathcal{M}\to \mathcal{A}_{\ell}\otimes \mathcal{L}_{\ell}\otimes K_C^{p-1}\to0,
    $$
    where $\mathcal{N}=\mathcal{M}\cap F^*\mathcal{E}'$ is a $\nabla$-invariant subbundle of $F^*\mathcal{E}'$. Claim that $\mathcal{N}\neq0$. Otherwise, we have $\mathcal{M}\cong\mathcal{A}_{\ell}\otimes \mathcal{L}_{\ell}\otimes K_C^{p-1}$. As $\deg(\mathcal{A}_{\ell})=\ell>0$, $\mathcal{A}_{\ell}\otimes\mathcal{L}_{\ell}\otimes K_C^{p-1}$ is actually the maximal destabilizer of $\mathcal{V}_{\ell}$. Therefore, $\mathcal{M}=\mathcal{A}_{\ell}\otimes\mathcal{L}_{\ell}\otimes K_C^{p-1}$. However, it is not $\nabla$-invariant by construction. Contradiction. Denote the rank of $\mathcal{N}$ by $n$. Now we know $0<n<p$. Again by Cartier descent, there is a unique rank $n$ subbundle $\widetilde{\mathcal{N}}$ of $\mathcal{E}'$ such that $\mathcal{N}=F^*\widetilde{\mathcal{N}}$. By Theorem \ref{thm: stronger inequality for slope bound}, we know that
    $$
\frac{\deg(\mathcal{N})}{np}=\frac{\deg(\widetilde{\mathcal{N}})}{n}\leq\frac{(g-1)}{p}\left(n-1\right)+\frac{\deg(\mathcal{L}_{\ell})}{p}.
    $$
Thus
    $$
    \begin{aligned}
\mu(\mathcal{M})&=\frac{\deg(\mathcal{N})+2(p-1)(g-1)+\deg(\mathcal{L}_{\ell})+\ell}{n+1}\\&\leq\frac{\left(n^2-n+2p-2\right)(g-1)+\ell}{n+1}+\deg(\mathcal{L}_{\ell}).
    \end{aligned}
    $$
Now we show if $0<n<p$ and $0<\ell\leq 2(g-1)$, then $\mu(\mathcal{M})\leq\mu(\mathcal{V}_{\ell})$. Indeed, it suffices to show
$$
\frac{\left(n^2-n+2p-2\right)(g-1)+\ell}{n+1}\leq \frac{p+2}{p+1}(p-1)(g-1)+\frac{\ell}{p+1},
$$
which is reduced to show
$$
\frac{(p-n)}{(p+1)(n+1)}(n(p+1)(g-1)-(p-1)(g-1)-\ell)\ge0.
$$
But this is clear. Furthermore, when $\ell<2(g-1)$, the above inequality is strict. It implies $(\mathcal{V}_{\ell},\nabla_{\ell})$ is $\nabla$-stable if $\ell<2(g-1)$. Finally we show that $(\mathcal{V}_\ell,\nabla_\ell)$ is also $\nabla$-stable when $\ell=2(g-1)$. Suppose $\ell=2(g-1)$. (In this case, $\mathcal{A}_{\ell}=K_C$.) The inequality above becomes an equality only when $\mathrm{rk}(\mathcal{N})=1$. By above, $\deg(\mathcal{N})\leq \deg(\mathcal{L}_\ell)$. Without lose of generality we may simply take $\mathcal{L}_{\ell}=\mathcal{O}_C$ in the case $\ell=2(g-1)$. If $\deg(\mathcal{N})<0$, then 
$$
\mu(\mathcal{M})<p(g-1)=\mu(\mathcal{V}_{2(g-1)}).
$$
So we may assume $\deg(\mathcal{N})=0$. Notice that $\mathcal{M}$ cannot contain $K_C^p$: If it were the case, since
$$
K_C^p\stackrel{\nabla|_{K_C^p}}{\longrightarrow}(K_C^p\oplus K_C^{p-1})\otimes K_C\twoheadrightarrow K_C^{p-1}\otimes K_C=K_C^p
$$
is an isomorphism, $\mathcal{M}$ would contain $K_C^p\oplus K_C^{p-1}$. Then, since
$$
K_C^{p-1}\stackrel{\nabla|_{K_C^{p-1}}}{\longrightarrow}\mathcal{N}^{p-2}\otimes K_C\twoheadrightarrow K_C^{p-2}\otimes K_C=K_C^{p-1}
$$
is again an isomorphism by Lemma \ref{lemma: Canonical filtration on curves}, where $\mathcal{N}^{\bullet}$ is the canonical filtration on $F^*\mathcal{E}'$. This shows $\mathcal{M}$ would contain $K_C^p\oplus\mathcal{N}^{p-2}$, which is impossible. Thus, if we define a filtration on $(\mathcal{V}_{2(g-1)},\nabla_{2(g-1)})$ by
$$
0\subset K_C^p\subset K_C^p\oplus\mathcal{N}^{\bullet},
$$
it follows that
$$
\mathrm{Gr}(\mathcal{M},\nabla_{2(g-1)}|_{\mathcal{M}})\subset\mathrm{Gr}(F^*\mathcal{E}',\nabla_{\can})
$$
where $\mathrm{Gr}$ denotes the graded bundle associated to the above filtration. Since $\mu(\mathcal{M})=p(g-1)>(p-1)(g-1)=\mu(F^*\mathcal{E}')$, we derive a contradiction to the stability of $\mathrm{Gr}(F^*\mathcal{E}',\nabla_{\can})$ as asserted in Lemma \ref{lemma: stability of Higgs bundle}. It implies that such a $\nabla$-invariant rank $2$ subbundle $\mathcal{M}$ which is an extension of $K_C^p$ by a degree zero line bundle $\mathcal{N}\subset F^*\mathcal{E}'$ does not exist. So $(\mathcal{V}_{2(g-1)},\nabla_{2(g-1)})$ is also $\nabla$-stable. This completes the proof.
\end{itemize}
\end{proof}
It is clear that the exponent of nilpotency of $p$-curvature of $(\mathcal{V}_{\ell},\nabla_{\ell})$ is $\leq 1$, since $0\subset F^*\mathcal{E}'\subset\mathcal{V}_{\ell}$ is a filtration of $\mathcal{V}_{\ell}$ such that the $p$-curvature of each graded piece is zero. Thus $(\mathcal{E}_{\ell},\theta_{\ell})=C(\mathcal{V}_{\ell},\nabla_{\ell})$ is well-defined, and it is a stable Higgs bundle as followed from Lemma \ref{lemma: V,nalba is stable}.
\begin{prop}\label{prop: counter-example of the Lan-Sheng-Zuo conjecture of rank p+1 on every p over curve}
For integer $\ell$ with $0<\ell\leq 2(g-1)$, the nilpotent stable rank $p+1$ Higgs bundle $(\mathcal{E}_{\ell},\theta_{\ell})$ of exponent $\leq 1$ constructed as above is not strongly semistable.
\end{prop}
\begin{proof}
Note that a Griffiths transverse saturated filtration on $(\mathcal{V}_{\ell},\nabla_{\ell})=C^{-1}(\mathcal{E}_{\ell},\theta_{\ell})$ is given by
$$
0\subset\mathcal{A}_{\ell}\otimes \mathcal{L}_{\ell}\otimes K_C^{p-1}\subset\mathcal{A}_{\ell}\otimes \mathcal{L}_{\ell}\otimes K_C^{p-1}\oplus\mathcal{N}^{\bullet},
$$
where $\mathcal{N}^{\bullet}$ is the canonical filtration of $F^*\mathcal{E}'$. Moreover, this filtration of $(\mathcal{V}_{\ell},\nabla_{\ell})$ is the unique gr-semistable one since the associated graded Higgs bundle is 
$$
(\mathcal{A}_{\ell}\otimes K_C^{p-1}\oplus K_C^{p-1}\oplus K_C^{p-2}\oplus\dots\oplus\mathcal{O}_C,\theta\oplus\bigoplus_{i=1}^{p-1}\theta_i)\otimes(\mathcal{L}_{\ell},0),
$$
which is a stable Higgs bundle with the exponent of nilpotency $=p$ by the same argument as Lemma \ref{lemma: stability of Higgs bundle}. By Lemma \ref{lemma: polystable representative has minimal exponent}, there is no other gr-semistable filtration on $(\mathcal{V}_{\ell},\nabla_{\ell})$ such that the exponent of nilpotency of the associated graded Higgs bundle is $\leq p-1$. This completes the proof.
\end{proof}
\begin{cor}\label{cor: cor for all big rank}
Let $C$ be a curve of genus $g\ge2$. There exists a nilpotent semistable Higgs bundle $(\mathcal{E},\theta)$ of exponent $\leq1$ and any big rank which is not strongly semistable.
\end{cor}
\begin{proof}
Set $(\mathcal{E},\theta)=(\mathcal{E}_{2(g-1)},\theta_{2(g-1)})$. Then $\mu(\mathcal{E})$ is an integer. Let $\mathcal{M}$ be a line bundle whose slope is the same as the slope of $(\mathcal{E},\theta)$. We know that $(\widetilde{\mathcal{E}},\widetilde{\theta})=(\mathcal{E},\theta)\oplus(\mathcal{M},0)$ is a nilpotent semistable Higgs bundle of exponent $\leq1$. Since the inverse Cartier transform preserves direct sum,  
$$
C^{-1}(\widetilde{\mathcal{E}},\widetilde{\theta})= C^{-1}(\mathcal{E},\theta)\oplus (F^*\mathcal{M},\nabla_{\can}).
$$
Then 
$$
0\subset F^*\mathcal{M}\subset F^*\mathcal{M}\oplus\mathcal{N}^{\bullet},
$$
is a gr-polystable filtration on $C^{-1}(\widetilde{\mathcal{E}},\widetilde{\theta})$, where $\mathcal{N}^{\bullet}$ is the gr-semistable filtration on $C^{-1}(\mathcal{E},\theta)$, and the exponent of nilpotency of the associated graded Higgs bundle is $p$. By Lemma \ref{lemma: polystable representative has minimal exponent} again, there is no gr-semistable filtration on it such that the exponent of nilpotency of the associated graded Higgs bundle is $\leq p-1$. This produces a semistable but not strongly semistable Higgs bundle of rank $p+2$. Repeating this process yields a semistable but not strongly semistable Higgs bundle of any big rank.
\end{proof}
Next, we provide a construction of two strongly semistable Higgs bundles whose tensor product is semistable but not strongly semistable. It is based on the existence of a uniformizing flat bundle with \emph{nilpotent} $p$-curvature. 

From now on, we use $K_C^{\frac12}$ to denote a fixed line bundle such that its square is isomorphic to $K_C$.
\begin{defn}\label{defn: uniformizing flat bundle}
A flat bundle $(\mathcal{V}_{\unif},\nabla_{\unif})$ together with a Griffiths transverse filtration is called a \emph{uniformizing flat bundle}, if the associated graded Higgs bundle is the uniformizing Higgs bundle $(\mathcal{E}_{\unif},\theta_{\unif}):=(K_C^{\frac12}\oplus K_C^{-\frac12},\theta)$, where $\theta\colon K_C^{\frac12}\to K_C^{-\frac12}\otimes K_C$ is an isomorphism.
\end{defn}

It is not difficult to construct a uniformizing flat bundle, but it seems difficult to construct one with nilpotent $p$-curvature. The following result provides a partial solution.
\begin{thm}[\protect{\cite[Theorem 2.1]{MR3905130}}]\label{thm: existence of uni flat bundle with nilpotent p-curvature}
Assume $p$ odd and $C$ to be generic. Then there exists a $W_2(k)$-lifting $C\hookrightarrow\widetilde{C}$ such that the uniformizing Higgs bundle with respect to the $C\hookrightarrow\widetilde{C}$ is two-periodic. 
More precisely, there is an isomorphism of graded Higgs bundles
$$
(\mathrm{Gr}_{\mathrm{HN}}\circ C^{-1})(\mathcal{E}_{\unif},\theta_{\unif})\cong (\mathcal{E}_{\unif},\theta_{\unif})\otimes(\mathcal{L},0),
$$
where $\mathrm{HN}$ denotes for the Harder-Narasimhan filtration, and $\mathcal{L}$ is a torsion line bundle of order two.
\begin{remk}
To be precise, in the statement of \cite[Theorem 2.1]{MR3905130}, the authors only mention the existence of such filtration, but it can be seen from the outline of proof that it is the Harder-Narasimhan filtration: In fact, the filtration is given as $\text{Fil}_1$ in the proof of Proposition $3.6$ in \cite{MR3905130}.
\end{remk}

\end{thm}

Put $(\mathcal{V}_{\unif},\nabla_{\unif})=C^{-1}\circ \mathrm{Gr}_{\text{HN}}\circ C^{-1}(\mathcal{E}_{\unif},\theta_{\unif})$. Then by Theorem \ref{thm: existence of uni flat bundle with nilpotent p-curvature}, it is a uniformizing flat bundle with nilpotent $p$-curvature. For the case $p=2$, the above theorem does not work. However, there is a direct way to construct a uniformizing flat bundle when $g(C)$ is odd: Since both $K_C\cong F^*K_C^{\frac12}$ and $\mathcal{O}_C$ can be equipped with the canonical connection, we define $\mathcal{V}=K_C\oplus\mathcal{O}_C$ and equip it with a connection $\nabla$ defined by
\begin{equation}\label{eq: p=2 uniformizing flat bundle}
\begin{pmatrix}
\nabla_{\can}&\theta\\ 
0&\nabla_{\can}
\end{pmatrix},
\end{equation}
where $\theta\colon K_C\to\mathcal{O}_C\otimes K_C$ is an isomorphism. Then $0\subset K_C\subset\mathcal{V}$ is the unique gr-semistable filtration on $(\mathcal{V},\nabla)$ such that the associated graded Higgs bundle is $(\mathcal{E}_{\unif},\theta_{\unif})\otimes(K_C^{\frac12},0)$. When $g(C)$ is odd, there exists some line bundle $\mathcal{L}$ such that $F^*\mathcal{L}=K_C^{\frac12}$. Then $(\mathcal{V},\nabla)\otimes (K_C^{-\frac12},\nabla_{\can})$ is a uniformizing flat bundle with nilpotent $p$-curvature.  

Now we can proceed to our last construction. First, we define a Higgs bundle $(\mathcal{E}_1,\theta_1)$ as follows: When $p=2$, we set $(\mathcal{E}_1,\theta_1)=C(\mathcal{V},\nabla)$, where $(\mathcal{V},\nabla)$ is constructed as above; when $p$ is odd, we set $(\mathcal{E}_1,\theta_1)$ to be 
$$
C(\mathcal{V}_{\unif},\nabla_{\unif})\cong (\mathcal{E}_{\unif},\theta_{\unif})\otimes (\mathcal{L},0).
$$
Clearly, $(\mathcal{E}_1,\theta_1)$ is a nilpotent stable Higgs bundle of exponent $\leq1$. Next, let $\mathcal{E}_2=F_*K_C^{\frac{1-p}{2}}$ (note that $\mathcal{E}_2=F_*K_C^{-\frac12}$ when $p=2$). By Theorem \ref{thm: stable pushforward under Frobenius is stable} and Riemann-Roch theorem, $\mathcal{E}_2$ is a stable bundle of rank $p$ and degree zero. Moreover, by Lemma \ref{lemma: Canonical filtration on curves}, Lemma \ref{lemma: stability of Higgs bundle} and Lemma \ref{lemma: uniqueness of gr-semistable fil} the unique gr-semistable filtration on $F^*\mathcal{E}_2$ is given by its canonical filtration, and the associated graded Higgs bundle is 
$$
(K_C^{\frac{p-1}{2}}\oplus\dots\oplus\mathcal{O}_C\oplus\dots\oplus K_C^{\frac{1-p}{2}},\bigoplus_{i=1}^{p-1}\theta_i),
$$
where $\theta_i\colon K_C^{i}\to K_C^{i-1}\otimes K_C$ is an isomorphism for each $i=\frac{3-p}{2},\dots,\frac{p-1}{2}$. Third, we show that $(\mathcal{E}_1,\theta_1)\otimes (\mathcal{E}_2,0)$ is a stable Higgs bundle. For $p=2$, if $(\mathcal{F},\theta_{\mathcal{F}})$ is a proper Higgs subbundle of $(\mathcal{E}_1,\theta_1)\otimes (\mathcal{E}_2,0)$, its image in $(K_C\otimes\mathcal{E}_2,0)$ would be of rank $\leq 1$ by the argument similar to Lemma \ref{lemma: stability of Higgs bundle} and then we can check that the slope of $\mathcal{F}$ is smaller than that of $\mathcal{E}_1\otimes\mathcal{E}_2$. For $p$ odd, we have explicitly   
$$
(\mathcal{E}_1,\theta_1)\otimes(\mathcal{E}_2,0)\cong \left((\mathcal{E}_2\otimes K_C^{\frac12})\oplus(\mathcal{E}_2\otimes K_C^{-\frac12}),\id\otimes\theta_1\right)\otimes(\mathcal{L},0),
$$ 
which is a stable Higgs bundle: Indeed, if $(\mathcal{F},\theta_{\mathcal{F}})$ is a proper subbundle and if $r$ (resp. $s$) denotes the rank of the image (resp. kernel) of the projection from $(\mathcal{F},\theta_{\mathcal{F}})$ to $(\mathcal{E}_2\otimes K_C^{\frac12},0)$, an argument similar to Lemma \ref{lemma: stability of Higgs bundle} implies the inequality $r\leq s$, and it implies that the slope of $\mathcal{F}$ is less than $0$. Fourth, there is the following diagram of the Higgs-de Rham flow operator acting upon $(\mathcal{E}_1,\theta_1)\otimes (\mathcal{E}_2,0)$:
\begin{center}
\begin{tikzcd}
                                                                        & {C^{-1}(\mathcal{E}_1,\theta_1)\otimes (F^*F_*K_C^{\frac{1-p}{2}},\nabla_{\can})} \arrow[rd, "\mathrm{Gr}_{\mathcal{N}^{\bullet}}"] &                            \\
{(\mathcal{E}_1,\theta_1)\otimes(\mathcal{E}_2,0)} \arrow[ru, "C^{-1}"] &                                                                                                                            & {(\mathcal{E}',\theta')},
\end{tikzcd}
\end{center}
where $\mathcal{N}^{\bullet}$ is the filtration given by tensor product of the gr-semistable filtration on $C^{-1}(\mathcal{E}_1,\theta_1)$ and the canonical filtration on $F^*F_*K_C^{\frac{1-p}{2}}$. Thus $$(\mathcal{E}',\theta')\cong 
\begin{cases}
(K_C^{\frac{p-1}{2}}\oplus\dots\oplus K_C^{\frac{1-p}{2}},\bigoplus_{i=1}^{p-1}\theta_i)\otimes(\mathcal{E}_{\unif},\theta_{\unif})\otimes(K_C^{\frac12},0),& p=2;\\
(K_C^{\frac{p-1}{2}}\oplus\dots\oplus K_C^{\frac{1-p}{2}},\bigoplus_{i=1}^{p-1}\theta_i)\otimes(\mathcal{E}_{\unif},\theta_{\unif}),& \text{$p$ odd}. 
\end{cases}
$$
By Lemma \ref{lemma: E_k,theta_k is polystable} below, $(\mathcal{E}',\theta')$ is polystable whose exponent of nilpotency is exactly $p$. Hence $\mathcal{N}^{\bullet}$ is a gr-polystable filtration. It follows from Lemma \ref{lemma: polystable representative has minimal exponent} that there is no gr-semistable filtration on $C^{-1}(\mathcal{E}_1,\theta_1)\otimes (F^*F_*K_C^{\frac{1-p}{2}},\nabla_{\can})$ such that its associated graded Higgs bundle has exponent of nilpotency $\leq p-1$. This shows that $(\mathcal{E}_1,\theta_1)\otimes(\mathcal{E}_2,0)$ is \emph{not} strongly semistable. Fifth, by the strong semistability theorem, both $(\mathcal{E}_1,\theta_1)$ and $(\mathcal{E}_2,0)$ are strongly semistable. Summarizing the above altogether, we have constructed an example of strongly semistable Higgs bundles whose tensor product is semistable but not strongly semistable, as claimed.

\begin{lemma}\label{lemma: E_k,theta_k is polystable}
For any $m\in\Z_{>0}$, 
$$
(\mathcal{E}^m,\theta^m):=(K_C^{\frac{m-1}{2}}\oplus\dots\oplus K_C^{\frac{1-m}{2}},\bigoplus_{i=1}^{m-1}\theta_i)\otimes(\mathcal{E}_{\unif},\theta_{\unif})
$$
is a polystable Higgs bundle which is nilpotent of exponent $=m$.
\end{lemma}
\begin{proof}
By the same argument as Lemma \ref{lemma: stability of Higgs bundle}, we know that the $m$-th symmetric product of the uniformizing Higgs bundle
$$
\mathrm{Sym}^{m}(\mathcal{E}_{\unif},\theta_{\unif})=(K_C^{\frac{m}{2}}\oplus\dots\oplus K_C^{-\frac{m}{2}},\bigoplus_{i=1}^{m}\theta_i)
$$
is stable. 

Note that $(\mathcal{E}_{\unif},\theta_{\unif})$ is an $\SL_2$-principal Higgs bundle together with the standard representation $V$ of $\SL_2$. By a basic fact on finite dimensional representations of $\SL_2$ (see e.g \cite[Exercise 11.11*]{MR1153249}), we have 
$$
\mathrm{Sym}^{m-1}V\otimes V=\mathrm{Sym}^{m}V\oplus\mathrm{Sym}^{m-2}V,
$$
and thus
$$
\begin{aligned}
(\mathcal{E}^m,\theta^m)&=\mathrm{Sym}^{m-1}(\mathcal{E}_{\unif},\theta_{\unif})\otimes(\mathcal{E}_{\unif},\theta_{\unif})\\ 
&=\mathrm{Sym}^{m}(\mathcal{E}_{\unif},\theta_{\unif})\oplus\mathrm{Sym}^{m-2}(\mathcal{E}_{\unif},\theta_{\unif}),
\end{aligned}
$$
which is a polystable Higgs bundle which is nilpotent of exponent $=m$.
\end{proof}

\bibliographystyle{alpha}
\bibliography{reference}
\end{document}